\documentclass[10pt]{amsart}
\usepackage{amssymb,amsmath,exscale,latexsym,amsthm,graphics}
\usepackage{epsfig}    
\usepackage{epic}      
\usepackage{eepic} 
\usepackage[matrix,arrow,curve,cmtip]{xy} 

\usepackage{overpic}

\theoremstyle{plain}
        \newtheorem{theorem}{Theorem}[section]
        \newtheorem*{theorem*}{Theorem}
        \newtheorem*{conj*}{Conjecture}
        \newtheorem{lemma}[theorem]{Lemma}
        
        \newtheorem{cor}[theorem]{Corollary}

\theoremstyle{definition}

\theoremstyle{remark}

\numberwithin{equation}{section}

\newcounter{mylistnum}

\newcommand{\dist}{\operatorname{dist}}

\newcommand{\diam}  {\operatorname{diam}}

\newcommand{\id} {\operatorname{id}}
\newcommand{\R}{\mathbb{R}}  
\newcommand{\C}{\mathbb{C}}      
\newcommand{\N}{\mathbb{N}}      

\providecommand{\abs}[1]{\lvert#1\rvert}

\newcommand{\I} {\mathbf{I}}

\newcommand{\G} {\mathbf{\Gamma}}

\newcommand{\Sim} {\!\sim \!}

\newcommand{\Se}{\mathsf{S}^1}

\newcommand{\dia}{\operatorname{\mathsf{dd}}}

\begin{document}
\title{Bounded turning circles are weak-quasicircles}

\author{Daniel Meyer}

\thanks{This research was supported by the Academy of Finland,
  projects SA-134757 and SA-118634} 
\address{University of Helsinki, Department of Mathematics and
  Statistics, P.O. Box 68,
FI-00014 University of Helsinki, Finland.}
\email{dmeyermail@gmail.com}

\date{\today}


\begin{abstract}  
  We show that a metric Jordan curve $\Gamma$ is \emph{bounded
    turning} if and only if there exists a \emph{weak-quasisymmetric}
  homeomorphism $\varphi\colon \mathsf{S}^1 \to \Gamma$.  
\end{abstract} 
\maketitle

\section{Introduction}
\label{sec:introduction}

A metric Jordan curve $\Gamma$ is \emph{bounded turning} (or
$C$-bounded turning) if there is
a constant $C\geq 1$ such that for each pair of points $x,y\in \Gamma$,
the arc of smaller diameter $\Gamma[x,y]\subset \Gamma$ between $x,y$
satisfies
\begin{equation}
  \label{eq:bt}
  \diam \Gamma[x,y]\leq C \abs{x-y}.
\end{equation}
Here and in the following, we denote metrics by the \emph{Polish
  notation}, i.e., by $\abs{x-y}$. A homeomorphism of
metric spaces $\varphi\colon X \to Y$ is called a
\emph{weak-quasisymmetry} (or $H$-weak-quasisymmetry), if there is a
constant $H\geq 1$ such that 
\begin{equation}
  \label{eq:w-qs}
  \abs{x-y} \leq \abs{x-z} \quad \Rightarrow \quad \abs{f(x) - f(y)} \leq H
  \abs{f(x) -f(z)},
\end{equation}
for all $x,y,z\in X$. In the present paper, we prove the following
theorem.

\begin{theorem}
  \label{thm:main}
  A metric Jordan curve $\Gamma$ is bounded turning if and only if there
  exists a weak-quasisymmetric homeomorphism
  $\varphi \colon \Se \to \Gamma$. 
\end{theorem}

The same proof shows the following.

\begin{cor}
  A metric Jordan arc $A$ is bounded turning if and only if there is a
  weak-quasisymmetric homeomorphism $\varphi\colon [0,1] \to A$. 
\end{cor}

\subsection{Background}
\label{sec:background}

The following notion is closely related to weak-quasi\-sym\-me\-try. A
homeomorphism $\varphi\colon X\to Y$ of metric spaces is called a
\emph{quasisymmetry} if there exists a homeomorphism $\eta\colon
[0,\infty) \to [0,\infty)$ such that
\begin{equation}
  \label{eq:qs}
  \abs{x-y} \leq t\abs{x-z} \quad \Rightarrow \quad \abs{\varphi(x)
    -\varphi(y)} \leq \eta(t) \abs{\varphi(x) - \varphi(z)},
\end{equation}
for all points $x,y,z\in X$ and $t\in [0,\infty)$.
General background
on (weak-)quasi\-sym\-me\-tries can be found in \cite{JuhAn}.

Every quasisymmetry is a weak-quasisymmetry (pick $H= \eta(1)$). While
the reverse does not hold in general, it is true in many practically
relevant situations. 
Recall that a metric space is \emph{doubling} if there is a constant
$N$, such that every ball of radius $r$ can be covered by at most $N$
balls of radius $r/2$. Note that every Jordan curve $\Gamma\subset
\R^n$ is doubling.

\begin{theorem}[{\cite[Theorem 10.19]{JuhAn}}]
  \label{thm:wqs_qs_doubling}
  If $X$ is {connected} and both $X,Y$ are
  {doubling}, then every weak-quasisymmetry $\varphi\colon X\to
  Y$ is quasisymmetric
\end{theorem}

Definition (\ref{eq:qs}) for quasisymmetry appears in
\cite{TV}. In earlier work (for example in \cite{AhlBeurext},
\cite{MR27:4921}) quasisymmetry is defined by
(\ref{eq:w-qs}); it is however only applied to maps where the two
notions agree by the theorem cited above.

\smallskip
A \emph{quasicircle} is the image of the unit circle $\Se$ by a
quasisymmetric map. 
Ahlfors has given in \cite{MR27:4921} the following geometric
characterization for planar 
quasicircles. For a Jordan curve $\Gamma\subset \C$ it holds 
\begin{equation*}
  \text{$\Gamma$ is a quasicircle} 
  \Leftrightarrow 
  \Gamma \text{ is bounded turning.}
\end{equation*}
Tukia and V\"{a}is\"{a}l\"{a} generalize this characterization to all
metric 
Jordan curves in \cite{TV}, namely for a metric Jordan curve $\Gamma$ it
holds 
\begin{equation*}
  \text{$\Gamma$ is a quasicircle} 
  \Leftrightarrow 
  \Gamma \text{ is bounded turning and doubling.}
\end{equation*}

If we call the weak-quasisymmetric image of the unit circle $\Se$ a
\emph{weak-quasicircle}, then Theorem \ref{thm:main} may be expressed as
follows. For a Jordan curve $\Gamma$ it holds
\begin{equation*}
    \text{$\Gamma$ is a weak-quasicircle} \Leftrightarrow \Gamma \text{ is
    bounded turning.}
\end{equation*}
It is easy to see that the quasisymmetric image of a doubling space is
doubling (see \cite[Theorem 10.18]{JuhAn}). Thus one recovers from
Theorem \ref{thm:main} together with Theorem \ref{thm:wqs_qs_doubling}
the Tukia-V\"{a}is\"{a}l\"{a} characterization of quasicircles. 

The first example of a bounded turning circle that is not a
quasicircle was given by Tukia-V\"{a}is\"{a}l\"{a} in \cite[Example
4.12]{TV}. A simple catalog $\mathcal{S}$ of bounded turning circles
that includes a 
bi-Lipschitz copy of any bounded turning circle is given in
\cite{qc-bt-biLip}. A curve $S\in \mathcal{S}$ from this catalog is
doubling, i.e., a quasicircle, if and only if a simple condition is
satisfied. 

\subsection{Organization of the paper}
\label{sec:organization-paper}

The ``if''-part of Theorem \ref{thm:main} is trivial. Namely let
$\varphi\colon \Se \to \Gamma$ be $H$-weak-quasisymmetric. Consider
arbitrary points $a,b\in\Se$, and let
$[a,b]\subset \Se=[0,1]/\{0\sim 1\}$ be the arc between $a$ and
$b$ of smaller diameter. Then for points $x,y\in [a,b]$
it holds
\begin{align}
  \label{eq:wqs_then_bt}
  &\abs{\varphi(x) -\varphi(y)} \leq \abs{\varphi(x) -\varphi(a)} +
  \abs{\varphi(a) -\varphi(y)} \leq 2H \abs{\varphi(a) -\varphi(b)}. 
  \intertext{or, assuming $a\leq x\leq y\leq b$,}
  \notag
  &\abs{\varphi(x) -\varphi(y)} \leq H\abs{\varphi(x) -\varphi(b)} \leq
  H^2 \abs{\varphi(a) -\varphi(b)}.
\end{align}
Therefore $\diam \varphi([a,b]) \leq C \abs{\varphi(a) -\varphi(b)}$,
where $C=\min\{2H, H^2\}$. Thus $\Gamma$ is $C$-bounded turning. 

\smallskip
The rest of this paper concerns the construction of a
weak-quasisymmetry $\varphi \colon \Se\to  \Gamma$,
for a given bounded turning circle $\Gamma$. In Section
\ref{sec:preliminaries} we show that we can restrict our attention to
the case when $\Gamma$ is $1$-bounded turning. Also an elementary
lemma about dividing arcs into subarcs of equal diameter is proved. 

In Section \ref{sec:dividing_gamma} we divide $\Gamma$ into arcs
$\Gamma^n_1, \dots, \Gamma^n_{N^n}$ (for each $n\in \N$). Two arcs
$\Gamma^n_i,\Gamma^n_j$ have 
roughly the same diameter. Each arc $\Gamma^{n+1}_i$ is contained in a
(unique) arc $\Gamma^n_j$, thus the sets $\G^n=\{\Gamma^n_j \mid
j=1,\dots, N^n\}$ form \emph{subdivisions} of $\Gamma$. 

In Section \ref{sec:dividing-unit-circle} we divide the unit circle
$\Se$ into intervals $I^n_1,\dots, I^n_{N^n}$. Neighboring intervals
$I^n_j, I^n_{j+1}$ have roughly the same diameter. Furthermore the
combinatorics of the subdivisions of $\Gamma$ and $\Se$ is the same,
namely $\Gamma^{n+1}_i \subset \Gamma^n_j \Leftrightarrow I^{n+1}_i
\subset I^n_j$. 

The map $\varphi\colon \Se \to \Gamma$ is defined in Section
\ref{sec:weak-quasisymmetry}, by mapping endpoints of
intervals $I^n_j$ to endpoints of corresponding arcs $\Gamma^n_j$.   

Section \ref{sec:estimating-intervals} and Section
\ref{sec:estimating-order} are preparations to prove the
weak-qua\-si\-sym\-me\-try of $\varphi$. Namely we show, that the diameter of
any interval in $\Se$ can be estimated in terms of the subdivision-intervals
$I^n_j$. Then we show that if $I^n_i, I^m_j$ are the largest
subdivision-intervals contained in adjacent intervals of the same
length, then $\abs{m-n}$ is bounded. 

Section \ref{sec:proof-theorem} finishes the proof of Theorem
\ref{thm:main}. 

\subsection{Notation}
\label{sec:notation}

The unit circle is denoted by $\Se$, which we identify with
$[0,1]/\{0 \sim 1\}$. The unit circle is thus equipped with the
orientation inherited from the real line. We always assume that $\Se$
is equipped with the arc-length metric denoted by $\lambda(s,t)$,
i.e., if $0\leq s\leq t\leq 1$, 
then 
\begin{equation}
  \label{eq:def_lambda}
  \lambda(s,t) = \min\{\abs{t-s} , \abs{s+ (1-t)}\}. 
\end{equation}
The diameter with respect to this metric 
of an interval $I \subset 
\Se = [0,1]/\{0\sim 1\}$ is denoted by $\abs{I}$. 
Note that $\abs{I}$ equals the Lebesgue measure of $I$ in the case when
$\abs{I}\leq \abs{\Se\setminus I}$. 

\section{Preliminaries}
\label{sec:preliminaries}


We first show that we can restrict our attention to $1$-bounded turning
circles. More precisely, we show that any bounded turning circle is
bi-Lipschitz equivalent to a $1$-bounded turning circle. 

Then we prove that any arc can be divided into subarcs of equal
diameter. 
\subsection{Diameter distance}
\label{sec:diameter-distance}

Given any metric Jordan curve or Jordan arc $\Gamma$ we define the
\emph{diameter distance} on $\Gamma$ by 
\begin{equation}
  \label{eq:def_dia}
  \dia(x,y) := \diam \Gamma[x,y],
\end{equation}
for all $x,y\in \Gamma$, where $\Gamma[x,y]\subset \Gamma$ is the arc
of smaller diameter between $x,y$. We record some properties of
$\dia$.

\begin{lemma}
  \label{lem:prop_dia}
  \mbox{}
  \begin{enumerate}
  \item 
    \label{item:dd_metric}
    $\dia$ is a metric on $\Gamma$.
  \item 
    \label{item:dd_biL}
    $\Gamma$ is $C$-bounded turning if and only if $\id\colon
    \Gamma \to (\Gamma, \dia)$ is $C$-bi-Lipschitz.
  \item 
    \label{item:dd_diam}
    For any arc $A\subset \Gamma$ it holds
    \begin{equation*}
      \diam_{\dia}A = \diam A. 
    \end{equation*}
    Here $\diam_{\dia}$ denotes the diameter with respect to $\dia$. 
  \item 
    \label{item:dd_1bt}
    $(\Gamma,\dia)$ is $1$-bounded turning. 
  \end{enumerate}
\end{lemma}

\begin{proof}

(\ref{item:dd_metric}) is elementary. 

To prove (\ref{item:dd_diam}), first observe that for all $x,y\in A$,
$|x-y|\le\dia(x,y)$,  
so $\diam A\le\diam_{\dia}A$.  Next, for all $x,y\in A$,
$\dia(x,y)\le\diam A$, so $\diam_{\dia} A \le\diam A$. 

\smallskip
In the following $\Gamma[x,y]\subset \Gamma$ will always denote the
arc of smaller diameter between points $x,y\in \Gamma$.
Property (\ref{item:dd_1bt}) follows directly from (\ref{item:dd_diam}), since
$\dia(x,y)=\diam\Gamma[x,y]=\diam_{\dia}\Gamma[x,y]$ for all
$x,y\in \Gamma$. 

\smallskip
It remains to establish (\ref{item:dd_biL}).  If $\Gamma$ is $C$-bounded
turning, then for all
$x,y\in\Gamma$ 
$$
  \dia(x,y)=\diam(\Gamma[x,y])\le C\, |x-y| \le C\,\dia(x,y).
$$
Thus the identity map $\id\colon \Gamma \to (\Gamma,\dia)$ is
$C$-bi-Lipschitz.  Conversely, if this map is $C$-bi-Lipschitz, 
then for all $x,y\in\Gamma$ 
$$
  \diam(\Gamma[x,y])=\diam_{\dia}(\Gamma[x,y])=\dia(x,y)\le C\, |x-y|.
$$
Therefore $(\Gamma,\abs{\cdot})$ is $C$-bounded turning.
\end{proof}

It is elementary that postcomposing a $H$-weak-quasisymmetry with an
$L$-bi-Lipschitz map yields a $HL^2$-weak-quasisymmetry. 

Assume we have constructed for a given bounded turning circle $\Gamma$
a weak-quasisymmetry $\varphi\colon \Se\to (\Gamma,\dia)$. Then the
composition $\Se \xrightarrow{\varphi} (\Gamma,\dia) \xrightarrow{\id}
\Gamma$ is the desired weak-quasisymmetric parametrization of
$\Gamma$.   
Thus to prove Theorem \ref{thm:main} it is enough to construct a
weak-quasisymmetry $\varphi\colon \Se\to \Gamma$ for any $1$-bounded
turning circle $\Gamma$. 


\subsection{Dividing arcs}  \label{s:dividing-arcs} 
Here we prove that any metric Jordan arc can be divided into any given number of subarcs each having exactly the same diameter.

The problem of finding points on a metric Jordan arc such that
consecutive points are at the same distance is a non-trivial problem.
In 1930 Menger gave a proof \cite[p.\ 487]{MR1512632}, that is short, 
simple, and natural; but wrong.  It was proved for arcs in Euclidean space
in \cite{61.0634.05}, and in the general case (indeed in more
generality) in \cite[Theorem 3]{66.0898.02}; see also
\cite{MR678135}. 

For the case at hand, i.e., for bounded turning arcs, it suffices to
find subarcs that have equal diameter.  We give the following
elementary proof for this problem. 

\begin{lemma}   \label{lem:dividing_arcs}
  Let $A$ be a metric Jordan arc and $N\geq 2$ an integer. Then we can divide $A$ into $N$ subarcs of equal diameter.
\end{lemma} 
\begin{proof}%
We may assume that $A$ is the unit interval $[0,1]$ equipped with some metric $d$.  We claim that there are points $0=s_0 < s_1 < \dots < s_{N-1} < s_N=1$ such that
\begin{equation*}
  \diam[s_0,s_1] = \diam [s_1,s_2] = \dots = \diam[s_{N-1},s_N]
\end{equation*}
where $\diam$ denotes diameter with respect to the metric $d$.  When
$N=2$ this follows by applying the intermediate value theorem to the
function $[0,1]\ni s\mapsto \diam [0,s] - \diam[s,1]$. 

According to Lemma \ref{lem:prop_dia} (\ref{item:dd_diam}), we may
measure the diameter with respect to the diameter distance. 
Thus, using Lemma \ref{lem:prop_dia} (\ref{item:dd_1bt}), we may
assume that $A$ is $1$-bounded turning, i.e., that
for any $[s,t]\subset [0,1]$ 
\begin{equation} \label{eq:d_diam}
  d(s,t) = \diam [s,t] \,.
\end{equation}

\smallskip
Next we modify $d$ to get a metric $d_\epsilon$ that is \emph{strictly increasing} in the sense that
\begin{equation} \label{eq:de_strictly_inc}
  [s,t] \subsetneq [s',t'] \subset[0,1] \implies d_\epsilon(s,t) <  d_\epsilon(s',t') \,.
\end{equation}
The crucial point here is the \emph{strict} inequality, which need not hold in general.
To this end, fix $\epsilon>0$ and for all $s,t\in [0,1]$ set
\begin{equation*}
  d_\epsilon(s,t) := d(s,t) +\epsilon\abs{t-s} \,.
\end{equation*}
Then from \eqref{eq:d_diam} it follows that
\begin{equation*}
  \diam_\epsilon [s,t] = \diam[s,t] + \epsilon\abs{t-s} =
  d_\epsilon(s,t) \,,
\end{equation*}
where $\diam_\epsilon$ denotes diameter with respect to
$d_\epsilon$.  This immediately implies \eqref{eq:de_strictly_inc}.

\smallskip
We now show that $[0,1]$ can be divided into $N$ subintervals of
equal $d_\epsilon$-diameter.
Consider the compact set $S:=\{\mathbf{s}=(s_1,\dots, s_{N-1}) \mid 0\leq s_1 \leq
\dots \leq s_{N-1}\leq 1\}$. Set $s_0:=0, s_N:=1$. The function
$\varphi \colon S\to \R$ defined by
\begin{equation*}
  \varphi(\mathbf{s}) := \max_{0\leq i\leq N-1} \diam_\epsilon [s_i,s_{i+1}] -
  \min_{0\leq j\leq N-1} \diam_\epsilon [s_j,s_{j+1}]
\end{equation*}
assumes a minimum on $S$. If this minimum is zero, we are
done. Otherwise, there are adjacent intervals $[s_{i-1}, s_{i}],
[s_{i}, s_{i+1}]$ that have different $d_\epsilon$-diameter. Using the
intermediate value theorem as before, we can find $s'_i\in [s_{i-1},
s_{i+1}]$ such that $\diam_\epsilon [s_{i-1},s'_i] =
\diam_\epsilon[s'_i, s_{i+1}]$. Then from (\ref{eq:de_strictly_inc}) it
follows that
\begin{equation*}
  \min_{0\leq j < N} \diam_\epsilon [s_j,s_{j+1}]
  <
  \diam_\epsilon [s_{i-1},s'_i]
  =
  \diam_\epsilon[s'_i, s_{i+1}]
  <
  \max_{0\leq i < N} \diam_\epsilon [s_i,s_{i+1}].
\end{equation*}
Applying this procedure to all subintervals of maximal
$d_\epsilon$-diameter we obtain a strictly smaller
minimum for the function $\varphi$, which is impossible. Thus the minimum must be zero,
and so we can subdivide $[0,1]$ into $N$ subintervals of equal
$d_\epsilon$-diameter.

\smallskip
Consider now a sequence $\epsilon_n\searrow 0$, as $n\to
\infty$. Let $s_1^n< \dots < s_{N-1}^n$ be the points that divide
$[0,1]$ into $N$
subintervals of equal diameter with respect to $d_{\epsilon_n}$. We
can assume that for all $1\leq j < N$, all points $s^n_j$ converge
to $s_j$ as $n\to\infty$. It follows that
for all $1\leq i,j < N$,
\begin{equation*}
  \diam[s_i,s_{i+1}] = \lim_{n\to\infty} \diam_{\epsilon_n} [s^n_i,s^n_{i+1}] =
  \lim_{n\to\infty} \diam_{\epsilon_n} [s^n_j,s^n_{j+1}] = \diam[s_j, s_{j+1}]
\end{equation*}
as desired.
\end{proof}%

The previous lemma is also true for metric Jordan curves $\Gamma$.  In
this case we are free to choose any point in $\Gamma$ to be an
endpoint of one of the subarcs.

\section{Dividing $\Gamma$}
\label{sec:dividing_gamma}

Consider a $1$-bounded turning metric Jordan curve $\Gamma$. 
We fix a point $a_0\in \Gamma$, and an orientation of $\Gamma$. 

For each $n\in \N$ we will divide $\Gamma$ into arcs
$\Gamma^n_1, \dots, \Gamma^n_{N^n}$, labeled consecutively on $\Gamma$,
such that $a_0$ is the common 
endpoint of $\Gamma^n_1, \Gamma^n_{N^n}$. The set of these arcs is
denoted by $\G^n$. Here and in the following the upper index $n$ will
denote the order of the subdivision. In particular $N^1, N^2, \dots,
N^n, \dots$ will be some (increasing) sequence of positive integers,
not a geometric sequence. 
\begin{lemma}
  \label{lem:divideG}
  There are divisions $\G^n$ of $\Gamma$ as above with the following
  properties. 
  \begin{enumerate}
  \item 
    \label{item:G1}
    $\G^{n+1}$ is a \emph{subdivision} of $\G^n$. This
    means that every $\Gamma^{n+1}\in \G^{n+1}$ is 
    contained in a (unique) $\Gamma^n\in \G^n$.
  \item 
    \label{item:G2}
    The diameters of the arcs of the $n$-th subdivision are
    comparable, more precisely
    \begin{equation*}
      \frac{1}{2} \leq \frac{\diam\Gamma\phantom{'}}{\diam \Gamma'} \leq 2,
    \end{equation*}
    for all $\Gamma, \Gamma' \in \G^n$.
  \item 
    \label{item:G3}
    The diameters of the $n$-th and the $(n+1)$-th subdivision are
    comparable, more precisely 
    \begin{equation*}
      \frac{1}{16}\diam \Gamma^n \leq \diam \Gamma^{n+1} \leq
      \frac{1}{4} \diam \Gamma^n,
    \end{equation*}
    for all $\Gamma^{n+1}\in \G^{n+1}$ and $\Gamma^n\in \G^n$. 
  \end{enumerate}
\end{lemma}
The last property implies that each arc $\Gamma^n\in \G^n$ is subdivided
into at least four arcs $\Gamma^{n+1} \in \G^{n+1}$.

Before we construct these divisions of $\Gamma$, i.e., prove the
previous lemma, we need some preparation.

\begin{lemma}
  \label{lem:divide_arcs}
  Let $A$ be a $1$-bounded turning arc, and let $0<\delta\leq \diam A$. For each
  $n$ we divide $A$ into $n$ arcs $A_1, \dots, A_n$ of equal diameter
  (see Lemma \ref{lem:dividing_arcs}). Let $n$ be the smallest integer
  such that $\diam A_1 =\diam 
  A_2 = \dots = \diam A_n \leq \delta$. Then $\diam A_j \geq \delta/2$
  for all $j=1,\dots, n$.  
\end{lemma}

\begin{proof}
  Let $n$ be as in the statement. If $n=1$, then $\delta = \diam A$,
  and there is nothing to prove.

  \smallskip
  Assume now that $n\geq 2$. Assume that the statement is false. Then the
  subarcs of equal diameter $A_1,\dots, A_n$ have common diameter
  $\diam A_j < \delta/2$.

  \smallskip
  {\it Claim.} Suppose $A$ is subdivided into $k$ subarcs $A'_1,\dots, A'_k$ of equal
  diameter greater than $\delta$. Then $2k+1 \leq n$. 
  \\
  Assuming the $A_i$ and the $A'_j$ are ordered in the same order
  along $A$, we see that one needs at least $A_1,A_2,A_3$ to cover
  $A'_1$. Similarly, at least the first five arcs $A_1,\dots, A_5$ are
  needed to cover $A'_1\cup A'_2$. Inducting over the arcs $A'_1,
  \dots, A'_k$ proves the claim.

  \smallskip
  We obtain a contradiction when we set $k=n-1$. 
\end{proof}

\begin{proof}
  [Proof of Lemma \ref{lem:divideG}]
  We start by dividing $\Gamma$ into arcs $\Gamma^1_1, \dots ,
  \Gamma^1_{N^1}$ of equal diameter, such that  $\diam \Gamma/8  \leq \diam
  \Gamma^1_j \leq \diam \Gamma/4$ for all $j=1,\dots, N^1$ using Lemma
  \ref{lem:dividing_arcs} and Lemma \ref{lem:divide_arcs}, for some
  $N^1\in \N$. Here
  $a_0$ is the common endpoint of $\Gamma^1_1$ and $\Gamma^1_{N^1}$.  
  
  \smallskip
  Assume $\Gamma$ has been divided into arcs $\Gamma^n_1, \dots,
  \Gamma^n_{N^n}$ satisfying Lemma \ref{lem:divideG}, in particular $1/2\leq
  \diam \Gamma^n_i/\diam \Gamma^n_j \leq 2$ for all $i,j\in\{1,\dots,
  N^n\}$. 
  Set $\delta= \frac{1}{4}\min_j \diam \Gamma^n_j$. Using Lemma \ref{lem:dividing_arcs} and
  Lemma \ref{lem:divide_arcs}
  we divide each arc $\Gamma^n=\Gamma^n_i$ into arcs
  $\Gamma^{n+1}_1, \dots, \Gamma^{n+1}_{N}$ (here $\Gamma^{n+1}_j =
  \Gamma^{n+1}_{i,j}$ and $N=N^n_i$) of equal diameter, such that
  \begin{equation*}
    \delta/2 \leq \diam \Gamma^{n+1}_1 = \dots = \diam
    \Gamma^{n+1}_{N} \leq \delta. 
  \end{equation*}
  
  Let $\Gamma^{n+1}_1,\dots, \Gamma^{n+1}_{N^{n+1}}$ be the set of all
  these arcs, labeled along $\Gamma$, such that $a_0$ is the common
  point of $\Gamma^{n+1}_1, \Gamma^{n+1}_{N^{n+1}}$. It is clear that
  these arcs satisfy the properties of Lemma \ref{lem:divideG}. 

  Thus the arcs $\Gamma^n_1, \dots, \Gamma^n_{N^n}$ have been
  constructed for all $n$.  
\end{proof}

\section{Dividing the unit circle}
\label{sec:dividing-unit-circle}

For each $n\in \N$ we divide the unit circle $\Se=[0,1]/\{0 \Sim 1\}$
into intervals 
$I^n_1, \dots, I^n_{N^n}$, labeled consecutively on $\Se$. The common
endpoint of $I^n_1$ and $I^n_{N^n}$ is $0$. The set of these intervals is
denoted by $\I^n$.   

\begin{lemma}
  \label{lem:propIn}
  There are divisions $\I^n$ of the unit circle $\Se$ as above
  satisfying the following. 
  \begin{enumerate}
  \item 
    \label{item:propIn1}
    $\I^{n+1}$ is a \emph{subdivision} of $\I^n$. This means that every
    $I^{n+1}\in \I^{n+1}$ is contained in a (unique) interval $I^n\in
    \I^n$.
    \setcounter{mylistnum}{\value{enumi}}
  \end{enumerate}
  Two adjacent intervals $I,I'\in \I^n$ are called \emph{neighbors}
  (i.e., $I= I^n_j, I'= I^n_{j+1}$). Note that neighbors are always
  elements of the \emph{same} subdivision $\I^n$. 
  \begin{enumerate}
    \setcounter{enumi}{\value{mylistnum}}
  \item 
    \label{item:propIn2}
    The diameter of neighboring
    intervals are \emph{comparable}, more precisely they agree or
    differ by the factor $2$,
    \begin{equation*}
      {\abs{I}}/{\abs{I'}} \in \{1/2, 1, 2\},
    \end{equation*}
    for all neighbors $I,I'$.
  \item 
    \label{item:propIn4}
    If $I^{n+1}_i \subset I^n_j$ then $\abs{I^{n+1}_i} \leq
    \abs{I^n_j}/4$, for all $i=1,\dots, N^{n+1}$, $j=1,\dots, N^n$. 
  \item 
    \label{item:propIn3}
    The subdivisions $\I^n$ have the \emph{same combinatorics} as the 
    subdivisions $\G^n$. Namely
    \begin{equation*}
      I^{n+1}_i \subset I^n_j \quad \Leftrightarrow \quad \Gamma^{n+1}_i \subset
      \Gamma^n_j,
    \end{equation*}
    for all $i= 1,\dots N^{n+1}$, $j=1,\dots N^n$. 
  \end{enumerate}
\end{lemma}

\begin{proof}
Let $I=I^n_i$ be given. Assume the corresponding arc
$\Gamma^n=\Gamma^n_i$ is divided into $N=N^n_i$ arcs
$\Gamma^{n+1}_j$. Note that by construction $N^n_i\geq 4$. 

Let $c$ be the \emph{midpoint} of the interval $I$ (i.e., $c=
\frac{1}{2} (a+b)$ if $I=[a,b]$). It divides $I$ into the \emph{left}
and \emph{right half} of $I$.

\smallskip
To simplify the discussion we assume that $\abs{I}=1$. For the general
case, if we write in the following ``length of a subinterval is
$1/4$'', it has to be replaced by ``length of a subinterval is
$1/4\cdot \abs{I}$'' and so on. 

\smallskip
{\it Case 1.} $N$ is even.
\\
Starting from the left endpoint of $I$, we divide the left half of $I$
into intervals of length $1/4, 
1/8, \dots, 2^{-N/2}$ (times the length of $I$). There is one remaining
interval of length 
$2^{-N/2}$, which is the last interval of the left half of $I$. The right half of
the interval is divided in a symmetric fashion, meaning starting from
the right endpoint, we divide the right half into intervals of length
$1/4,1/8, \dots, 2^{-N/2+1}, 2^{-N/2}, 2^{-N/2}$. See the bottom of
Figure \ref{fig:1}. 

\begin{figure}
  \centering
  \begin{overpic}
    [width=11cm, 
    tics=20]
    {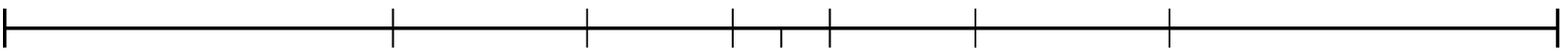}
    \put(12,-2){$\scriptstyle{{1}/{4}}$}
    \put(31,-2){$\scriptstyle{{1}/{8}}$}
    \put(68,-2){$\scriptstyle{{1}/{8}}$}
    \put(86,-2){$\scriptstyle{{1}/{4}}$}
    \put(41,-1){$\scriptstyle{\dots}$}
    \put(56.5,-1){$\scriptstyle{\dots}$}
    \put(47,-2.7){$\scriptstyle{{2^{-N/2}}}$}
    \put(49,4){$\scriptstyle{{2^{-m+1}}}$}
  \end{overpic}
  \caption{Subdividing an interval.}
  \label{fig:1}
\end{figure}

\smallskip
{\it Case 2.} $N=2m-1$ is odd. 
\\
We divide $I$ into $N+1=2m$ subintervals as in
Case 1. We then take the union of the two middle subintervals, i.e.,
the two subintervals containing the midpoint $c$. Thus $I$ is divided
into $N$ subintervals of lengths
\begin{equation*}
  1/4, 1/8, \dots, 2^{-m+1}, 2^{-m},\,
  2^{-m+1} \,, 2^{-m}, 2^{-m+1}, \dots, 1/8, 1/4.  
\end{equation*}
See the top of Figure \ref{fig:1}.


\smallskip
This finishes the division of $I$, thus of all $I^n_i$, into
intervals. Thus all $I^n_j$ have been constructed for all $n\in
\N$. It is clear that they satisfy the properties of Lemma
\ref{lem:propIn}. 

\smallskip
In Case 1 there are two subintervals of $I$ containing the midpoint of
$I$; 
in Case 2 there is a single subinterval of $I$. Such a subinterval is
called a \emph{middle} subinterval of $I$.  
\end{proof}
\section{The weak quasisymmetry}
\label{sec:weak-quasisymmetry}

Let $s^n_0, \dots, s^n_{N^n-1}$ be the endpoints of the intervals
$I^n_j$ ordered increasingly on $\Se=[0,1]/\{0\sim 1\}$,
$s^n_0=0$ for all $n\in \N$. Let 
$a^n_0,\dots, a^n_{N^n-1}$ be the endpoints of the arcs
$\Gamma^n_j$. Then we define $\varphi(s^n_j) = a^n_j$. From Lemma
\ref{lem:divideG} (\ref{item:G1}) and Lemma \ref{lem:propIn}
(\ref{item:propIn3}) it follows that $\varphi$ is well defined, i.e.,
if $s^n_i=s^m_j$ then $\varphi(s^n_i)= a^n_i = a^m_j =
\varphi(s^m_j)$. 

We show uniform continuity of $\varphi$ on the set $\mathbf{s}=\{s^n_j 
\mid n\in \N, j=0,\dots, N^n-1\}$. Let $\delta_n:=
\min_j \abs{I^n_j}$. Then if $\lambda(s, t) \leq \delta_n/2$, for two
points $s,t\in \mathbf{s}$ (recall from (\ref{eq:def_lambda}) that
$\lambda$ is the metric on $\Se$)
then $s, t$ are contained in adjacent intervals $I^n_{j},
I^n_{j+1}$. Thus $\varphi(s),\varphi(t)$ are contained in adjacent
arcs $\Gamma^n_j, \Gamma^n_{j+1}$. Thus 
\begin{equation*}
  \abs{\varphi(s)- \varphi(t)} \leq \diam \Gamma^n_j + \diam
  \Gamma^n_{j+1} \leq 2\cdot 4^{-n} \diam \Gamma,
\end{equation*}
by Lemma \ref{lem:divideG} (\ref{item:G3}), showing uniform continuity
of $\varphi$ on $\mathbf{s}$. Since this set is dense in $\Se$,
$\varphi$ extends continuously to $\Se$. The surjectivity is clear,
since the set $\{a^n_j \mid n\in \N, j=0,\dots, N^n-1\}$ is dense in
$\Gamma$. Injectivity follows from the fact that disjoint sets $I^n_i,
I^n_j$ are mapped to disjoint arcs $\Gamma^n_i, \Gamma^n_j$. Thus
$\varphi\colon \Se \to \Gamma$ is a homeomorphism.

\section{Estimating intervals}
\label{sec:estimating-intervals}

Given an interval $[x,y]\subset \Se$ we define
\begin{equation}
  \label{eq:def_delta}
  \delta([x,y]):= \max\{\abs{I^n_j} \mid I^n_j \subset [x,y]\}.
\end{equation}
Here the maximum is taken over $n\in \N$ and all intervals $I^n_j\in
\I^n$ as defined in 
Section \ref{sec:dividing-unit-circle}.

\begin{lemma}
  \label{lem:sizeI_delta}
  Let $[x,y]\subset \Se$ be any interval. Then
  \begin{equation*}
    \delta([x,y])\leq \abs{[x,y]} \leq 12\, \delta([x,y]). 
  \end{equation*}
  Furthermore, if the maximum in equation (\ref{eq:def_delta}) is
  attained for an interval $I=I^n_j \in \I^n$, then there are two
  intersecting (possibly identical)
  intervals $\hat{I}, \hat{J}\in \I^{n-1}$ such that 
  \begin{equation*}
    I\subset [x,y]\subset \hat{I} \cup \hat{J}. 
  \end{equation*}
\end{lemma}

\begin{proof}
  Let $I= I^n_j\subset [x,y]$ be one interval where the maximum from
  (\ref{eq:def_delta}) is attained, i.e., $\abs{I} =
  \delta([x,y])$. The left inequality, i.e., $\abs{I}
  =\delta([x,y])\leq \abs{[x,y]}$ is obvious. 
  Let $\hat{I}\supset I$ be the parent of $I$, i.e.,
  the unique interval $\hat{I}\in \I^{n-1}$ containing $I$. Assume
  that $\hat{I}$ was subdivided into $N$ intervals $I^n\in \I^n$.
  We consider several cases.  
  
  \smallskip
  {\it Case 1.} $\abs{I}=\abs{\hat{I}}/4$. 
  \\ 
  This can happen in three instances: either $I$ is the left- or rightmost
  interval in $\hat{I}$ (i.e., $I,\hat{I}$ share a boundary point);
  or $N$ is equal to $4$ or $5$, and $I$ contains the midpoint of
  $\hat{I}$.  
  
  If $[x,y]\subset \hat{I}$ we are done, since then $\abs{[x,y]} \leq
  \abs{\hat{I}} = 4 \abs{I} = 4 \delta([x,y])$. We set
  $\hat{J}:=\hat{I}$. 
  
  So assume that $[x,y]\not\subset \hat{I}$. This means that one
  endpoint of $\hat{I}$, without loss of generality the left endpoint,
  is an interior point of $[x,y]$. From the maximality of $I$ it
  follows that $y\in \hat{I}$. 
  Consider the left neighbor
  $\hat{J}\in \I^{n-1}$ of $\hat{I}$. Note that $\abs{\hat{J}} \geq \frac{1}{2}
  \abs{\hat{I} } = 2 \abs{I}$. Thus $\hat{J}\not \subset [x,y]$ by the
  maximality of $I$.  
  Thus
  $[x,y]\subset \hat{J} \cup \hat{I}$. It holds $\abs{\hat{I}} = 4
  \abs{I}$ and $\abs{\hat{J}} \leq  2\abs{\hat{I}} = 8 \abs{I}$ so
  \begin{equation*}
    \abs{[x,y]} \leq 12 \,\delta([x,y]). 
  \end{equation*} 
  
  \medskip
  {\it Case 2.} $N\geq 6$ is even, and $I=I^n_j$ is a middle subinterval of
  $\hat{I}$ (i.e., contains the midpoint of $I$).   
  \\
  Then either both $I^n_{j-2}, I^n_{j+3}$ or both $I^n_{j-3},I^n_{j+2}$ have
  diameter strictly bigger 
  than $I$. We can assume without loss of generality the former
  case. This means that $I$ is in the left half of $\hat{I}$ and that
  \begin{equation*}
    I^n_{j-2} \cup I^n_{j-1} \cup I^n_j \cup I^n_{j+1} \cup I^n_{j+2}
    \cup I^n_{j+3}
  \end{equation*}
  cover $[x,y]$. Note, that the total length of these sets is $8
  \abs{I}$. Thus $\delta([x,y]) \leq \abs{[x,y]} \leq 8
  \delta([x,y])$. 

  \medskip
  {\it Case 3.} $N\geq 7$ is odd, and $I^n_j$ is the middle
  subinterval of $\hat{I}$. 
  \\
  Similar to the preceding case, $I^n_{j-3}, I^n_{j+3}$ have
  twice the length as $I$, thus they are not contained in $[x,y]$ and  
  \begin{equation*}
    [x,y] \subset I^n_{j-3} \cup \dots \cup I^n_{j+3}
  \end{equation*}
  Note that the total length of these intervals is $8 \abs{I}$. This
  finishes the claim in this case.

  \medskip
  {\it Case 4.} Remaining case.
  \\
  One of the neighbors of $I=I^n_j$, without loss of generality the
  left neighbor $I^n_{j-1}$,
  has twice the length as $I$. 

  Furthermore, there is a subinterval $I^n_{j+k}\in \I^n$ of
  $\hat{I}$, that has the same length as $I$. It is symmetric to $I$ with
  respect to the midpoint of $\hat{I}$. Then $I^n_{j-1},
  I^n_{j+k+1}$ have twice the length of $I$, thus are not contained in
  $[x,y]$. Thus
  \begin{equation*}
    [x,y] \subset I^n_{j-1} \cup I^n_j \cup \dots \cup I^n_{j+k} \cup
    I^n_{j+k+1}. 
  \end{equation*}
  The total length of the right-hand side is $8
  \abs{I}$, finishing the claim. 

  \medskip
  Note that in Case 2--Case 4, the subintervals that cover $[x,y]$ are
  all contained in the parent $\hat{I}$, we then set $\hat{J}:= \hat{I}$.  
\end{proof}



\section{Estimating order}
\label{sec:estimating-order}

Consider now two adjacent intervals (in $\Se$) of the same length,
i.e., $[x-t,x], [x,x+t]$ for some $x\in \Se$ and $0< t\leq
1/2$. Consider the largest subdivision intervals contained in
$[x-t,x], [x,x+t]$, meaning we consider intervals $J^m\in \I^m,
I^n\in \I^n$ such that
\begin{align*}
  &J^m\subset [x-t,x], &&I^n \subset [x,x+t] \quad\text{and}
  \\
  &\abs{J^m} = \delta([x-t,x]),  && \abs{I^n} = \delta([x,x+t]). 
\end{align*}

We want to show that $n,m$ differ by at most a constant $k_0$ (in fact
$k_0=4$). Before giving the detailed argument, let us quickly describe
the idea. 
From Lemma \ref{lem:sizeI_delta} it follows that $\abs{I^n},
\abs{J^m}$ are comparable.  
Without loss of generality, we can assume that $n\leq m$. 
Let $J^n\in \I^n$ be the (unique) $n$-th order subdivision-interval
containing 
$J^m$. If $m-n$ is large, then $\abs{J^n}$ is large compared to
$\abs{J^m}$, thus large compared to $\abs{I^n}$. Then $J^n, I^n$
have to be far apart. This is impossible. 

\begin{lemma}
  \label{lem:mn_est}
  In the setting as above it holds that $\abs{m-n} \leq 4$. 
\end{lemma}

\begin{proof}
  As in the outline given above we assume that $n\leq m$, and let
  $J^n\in \I^n$ be the subdivision-interval containing $J^m$. If $m-n=
  k_0$, then 
  $\abs{J^m} \leq 4^{-k_0} \abs{J^n}$ by Lemma \ref{lem:propIn}
  (\ref{item:propIn4}).  

  \medskip
  {\it Claim.} Consider two intervals $I,I'\in \I^n$ such that
  $\abs{I'} / \abs{I} \geq 2^{i+1}$ for some $i\geq 1$. Then
  $\dist(I,I') \geq 2^i\abs{I}$. 

  \smallskip
  This is clear, since the interval between $I,I'$ has to contain one
  of size $2^i\abs{I}$ by Lemma \ref{lem:propIn}
  (\ref{item:propIn2}). 

  \medskip
  From Lemma \ref{lem:sizeI_delta} it follows that $\abs{[x-t,
    x+t]}\leq 24 \abs{I^n}$. Thus it follows from the previous claim 
  that $\abs{J^n} / \abs{I^n} \leq 2^5$. Indeed $\abs{J^n} / \abs{I^n}
  \geq 2^6$ implies by the claim that $\dist(J^n, I^n) \geq 2^5
  \abs{I^n} = 32 \abs{I^n}$, which is impossible. Thus by Lemma
  \ref{lem:sizeI_delta} 
  \begin{equation*}
    \frac{1}{12} \abs{I^n} \leq \frac{1}{12} \abs{[x,x+t]} =
    \frac{1}{12} \abs{[x-t,x]} \leq \abs{J^m} \leq 4^{-k_0}\abs{J^n}
    \leq 
    4^{-k_0} 2^5 \abs{I^n}.
  \end{equation*}
  We obtain a contradiction if we choose $k_0$ such that $4^{-k_0} 2^5
  < {1}/{12}$ or $k_0\geq 5$. This finishes the proof. 
\end{proof}

\section{Proof of the Theorem}
\label{sec:proof-theorem}

After these preparations, we are ready to prove the main theorem.

\begin{proof}[Proof of Theorem \ref{thm:main}] 
  Recall from Section \ref{sec:diameter-distance} that it is enough to
  prove the theorem in the case when
  $\Gamma$ is $1$-bounded turning. This means that for any two points
  $x,y\in \Gamma$, the arc of smaller diameter
  $\Gamma[x,y]\subset \Gamma$ between $x,y$ satisfies $\diam
  \Gamma[x,y] = \abs{x-y}$. 

  Thus it is enough to show that the arcs $\varphi([x-t,x])$,
  $\varphi([x,x+t])$ have comparable diameter for all $x\in \Se$,
  $0<t\leq 1/2$. Let $I_{-}\in \I^m,
  I_{+}\in \I^n$ be the largest intervals contained in $[x-t,x],
  [x,x+t]$, i.e., 
  \begin{align*}
    &I_{-} \subset [x-t,x], &&I_{+} \subset [x,x+t] \quad \text{and}
    \\
    &\abs{I_{-}} = \delta([x-t,x]), &&\abs{I_{+}} = \delta([x,x+t]).     
  \end{align*}
  Let $\hat{I}_{-}, \hat{J}_{-}\in \I^{m-1}$ be the intervals that
  cover $[x-t,x]$ according to Lemma \ref{lem:sizeI_delta}. Then
  \begin{align*}
    \abs{\varphi(x-t) - \varphi(x)} &= \diam \varphi([x-t, x]) \leq
    \diam \varphi(\hat{I}_{-} \cup \hat{J}_{-}) 
    \\
    &\leq 32 \diam \varphi(I_{-}) 
    \qquad\text{ by Lemma \ref{lem:divideG} (\ref{item:G3})}
    \\
    &\leq 32\cdot 2 \cdot 16^4 \diam \varphi(I_{+})  
    \intertext{by Lemma \ref{lem:mn_est}, Lemma \ref{lem:divideG}
      (\ref{item:G3}), and 
      Lemma \ref{lem:divideG} (\ref{item:G2}),}
    & \leq 32 \cdot 2\cdot 16^4 \diam \varphi([x,x+t])
    \\
    & = 32\cdot 2\cdot 16^4\,\abs{\varphi(x) - \varphi(x+t)}.
  \end{align*}
  This finishes the proof.
\end{proof}

\section{Concluding remarks}
\label{sec:concluding-remarks}

It is natural to ask, how small the involved
constants can be chosen. In particular, how small can the constant $H\geq 1$ of the
weak-quasisymmetric parametrization $\varphi\colon \Se\to \Gamma$ for
a given $C$-bounded turning circle be chosen? Recall from
(\ref{eq:wqs_then_bt}) that the image of the
unit circle by a $H$-weak-quasisymmetry is $C$-bounded turning, where
$C=\min\{2H, H^2\}$. Thus it is natural to ask, if any $C$-bounded
turning circle admits a $H$-weak-quasisymmetric parametrization, where
$H= \max\{C/2, \sqrt{C}\}$. As a starting point one may ask, if any
$1$-bounded turning circle admits a $1$-weak-quasisymmetric
parametrization. 

\section*{Acknowledgments}
\label{sec:acknowledgments}

The question that is answered by Theorem \ref{thm:main} was posed by
David Herron. Jussi V\"{a}is\"{a}l\"{a} provided many helpful
suggestions and references.

%
\bibliographystyle{amsalpha}
\bibliography{litlist}

\end{document}